\documentclass[11pt,a4paper]{article}

\usepackage[paper,margin=3cm]{paulmathdocs}
\usepackage{accents}

\newcommand{\thickbar}[1]{\accentset{\rule{0.45em}{0.6pt}}{#1}}

\title{Guaranteeing Hausdorff--Gromov--Hausdorff Equality for Metric Graphs}
\newcommand{\institution}{Karlsruher Institut für Technologie}
\author{Paul Schott\\\institution}
\date{}

\begin{document}

\maketitle

\begin{abstract}
    We prove the equality of Hausdorff and Gromov--Hausdorff distance \linebreak\mbox{\(\dGH(G, U) = \dH(G, U)\)} for a metric graph $G$ and a subset $U \subseteq G$, given that the Hausdorff distance between $G$ and $U$ is attained not just near the leaves of $G$ and the Gromov-Hausdorff distance is not too big compared to the graph's systole.
\end{abstract}

\section{Introduction}
    After its introduction in 1975 by David Edwards \cite{EdwardsStructureOfSuperspace} and, independently, in 1981 by Mikhail Gromov \cite{GromovStructuresMetriques}, the Gromov--Hausdorff distance was used as a canonical way of discussing convergence of metric spaces. For this, it is sufficient to merely bound the Gromov--Hausdorff distance from above or below. 
    However, especially in recent years, explicit calculation of Gromov-Hausdorff distances between spaces has become more important, with new applications arising, for example, in Computational Geometry \cite{SchmiedlComputationalAspects}, \cite{MemoliShapeComparison} but also in fields like Data Science \cite{CarlssonMemoliHierarchicalClustering}. 
    While canonical, the Gromov--Hausdorff distance is generally very hard to compute. 
    For abstractly described spaces like unit spheres of different dimensions, in \cite{LimMemoliSmithSpheres} the authors have managed to bound the Gromov--Hausdorff distance from both sides, sometimes even managing to bound by the same value.
    However, even for the general case of metric trees, approximating it within a factor of less than $3$ is NP-hard \cite{AgarwalComputingGHMetricTrees}. 
    In contrast, the closely related Hausdorff distance, given a surrounding space, can be calculated much more efficiently. The Hausdorff distance between two sets embedded in a common space is an upper bound for the Gromov-Hausdorff distance. However, as Adams et al. showed with a construction on metric graphs (see Theorem 3.1 in \cite{AdamsMajhiManinVirkZavaMetricGraphs}), it can become arbitrarily large, while the Gromov-Hausdorff distance becomes arbitrarily small, even in the setting of metric trees. 
    \begin{example} \label{Example:ConstructionForHMuchBiggerThanGH}
        Consider, for $\varepsilon > 0$ and $n \in \bbN$, $n \geq 2$, a tree $T$ resulting from the gluing of $n$ intervals of lengths $1 + \varepsilon, 1 + 2 \varepsilon, \ldots,1 + n \varepsilon$.
        \\ In it, we can isometrically embed a similar space $U_0$, resulting from the gluing of $n$ intervals with edge-lengths $1, 1 + \varepsilon, \ldots, 1 + (n-1) \varepsilon$, in two different ways, as indicated in \cref{fig:T-X0}.
        \\ The embedding on the left shows that $\dGH(T, U_0) \leq \dH(T, U_1) = \varepsilon$, while the embedding on the right shows $\dH(T, U_0) = n \cdot \varepsilon$.
        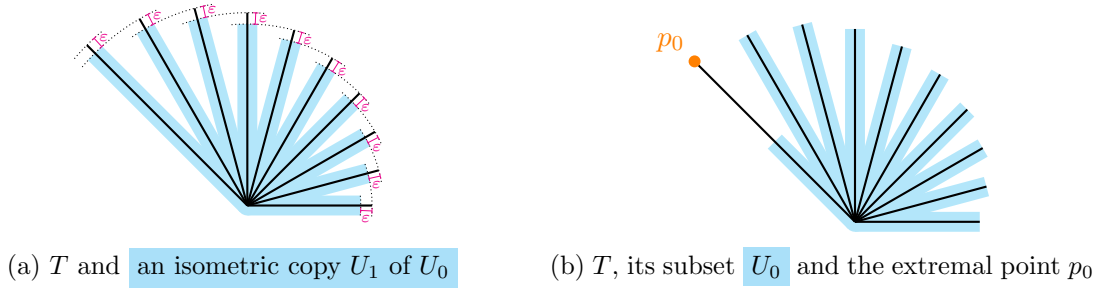
\begin{figure}[H]
            \centering
            \begin{subfigure}[t]{0.48\textwidth}
                \centering
                \begin{tikzpicture}[scale=1.5]
                    \begin{scope}[transparency group, cyan, opacity=0.3]
                        \draw[line width=7.5pt] (0, 0) -- (  0:1.0cm);
                        \draw[line width=7.5pt] (0, 0) -- ( 15:1.1cm);
                        \draw[line width=7.5pt] (0, 0) -- ( 30:1.2cm);
                        \draw[line width=7.5pt] (0, 0) -- ( 45:1.3cm);
                        \draw[line width=7.5pt] (0, 0) -- ( 60:1.4cm);
                        \draw[line width=7.5pt] (0, 0) -- ( 75:1.5cm);
                        \draw[line width=7.5pt] (0, 0) -- ( 90:1.6cm);
                        \draw[line width=7.5pt] (0, 0) -- (105:1.7cm);
                        \draw[line width=7.5pt] (0, 0) -- (120:1.8cm);
                        \draw[line width=7.5pt] (0, 0) -- (135:1.9cm);
                        \fill (0, 0) circle (2.5pt);
                    \end{scope}
                    \begin{scope}[thick]
                        \draw (0, 0) -- (  0:1.1cm);
                        \draw (0, 0) -- ( 15:1.2cm);
                        \draw (0, 0) -- ( 30:1.3cm);
                        \draw (0, 0) -- ( 45:1.4cm);
                        \draw (0, 0) -- ( 60:1.5cm);
                        \draw (0, 0) -- ( 75:1.6cm);
                        \draw (0, 0) -- ( 90:1.7cm);
                        \draw (0, 0) -- (105:1.8cm);
                        \draw (0, 0) -- (120:1.9cm);
                        \draw (0, 0) -- (135:2.0cm);
                    \end{scope}
                    % Each guide joins a black endpoint to the equally distant cyan endpoint on the neighboring spoke.
                    \begin{scope}[densely dotted, thin]
                        \draw (-5:1.0cm)
                            arc[start angle=-5, end angle=5, radius=1.0cm];
                        \foreach \angle/\radius in {
                            0/1.1, 15/1.2, 30/1.3, 45/1.4, 60/1.5,
                            75/1.6, 90/1.7, 105/1.8, 120/1.9
                        } {
                            \draw (\angle-5:\radius cm)
                                arc[start angle={\angle-5}, end angle={\angle+20},
                                    radius=\radius cm];
                        }
                        \draw (130:2.0cm)
                            arc[start angle=130, end angle=140, radius=2.0cm];
                    \end{scope}
                    % Fine radial dimensions mark the epsilon separation of every pair of neighboring arcs.
                    \begin{scope}[magenta, line width=0.2pt]
                        \foreach \angle/\inner/\middle/\outer in {
                            -2.5/1.0/1.05/1.1, 12.5/1.1/1.15/1.2,
                            27.5/1.2/1.25/1.3, 42.5/1.3/1.35/1.4,
                            57.5/1.4/1.45/1.5, 72.5/1.5/1.55/1.6,
                            87.5/1.6/1.65/1.7, 102.5/1.7/1.75/1.8,
                            117.5/1.8/1.85/1.9, 132.5/1.9/1.95/2.0
                        } {
                            \draw (\angle:\inner cm) -- (\angle:\outer cm);
                            \node[font=\tiny] at
                                ($ (\angle:\middle cm)
                                + (\angle-90:0.07cm) $)
                                {$\varepsilon$};
                            \draw
                                ($ (\angle:\inner cm) + (\angle+90:0.025cm) $)
                                --
                                ($ (\angle:\inner cm) + (\angle-90:0.025cm) $);
                            \draw
                                ($ (\angle:\outer cm) + (\angle+90:0.025cm) $)
                                --
                                ($ (\angle:\outer cm) + (\angle-90:0.025cm) $);
                        }
                    \end{scope}
                \end{tikzpicture}
                \caption{$T$ and \colorbox{cyan!30}{an isometric copy $U_1$ of $U_0$}}
                \label{fig:T-X0-copy}
            \end{subfigure}
           \hfill
            \begin{subfigure}[t]{0.48\textwidth}
                \centering
                \begin{tikzpicture}[scale=1.5]
                    \begin{scope}[transparency group, cyan, opacity=0.3]
                        \draw[line width=7.5pt] (0, 0) -- (  0:1.1cm);
                        \draw[line width=7.5pt] (0, 0) -- ( 15:1.2cm);
                        \draw[line width=7.5pt] (0, 0) -- ( 30:1.3cm);
                        \draw[line width=7.5pt] (0, 0) -- ( 45:1.4cm);
                        \draw[line width=7.5pt] (0, 0) -- ( 60:1.5cm);
                        \draw[line width=7.5pt] (0, 0) -- ( 75:1.6cm);
                        \draw[line width=7.5pt] (0, 0) -- ( 90:1.7cm);
                        \draw[line width=7.5pt] (0, 0) -- (105:1.8cm);
                        \draw[line width=7.5pt] (0, 0) -- (120:1.9cm);
                        \draw[line width=7.5pt] (0, 0) -- (135:1.0cm);
                        \fill (0, 0) circle (2.5pt);
                    \end{scope}
                    \begin{scope}[thick]
                        \draw (0, 0) -- (  0:1.1cm);
                        \draw (0, 0) -- ( 15:1.2cm);
                        \draw (0, 0) -- ( 30:1.3cm);
                        \draw (0, 0) -- ( 45:1.4cm);
                        \draw (0, 0) -- ( 60:1.5cm);
                        \draw (0, 0) -- ( 75:1.6cm);
                        \draw (0, 0) -- ( 90:1.7cm);
                        \draw (0, 0) -- (105:1.8cm);
                        \draw (0, 0) -- (120:1.9cm);
                        \draw (0, 0) -- (135:2.0cm);
                    \end{scope}
                    \fill[orange] (135:2.0cm) circle (1.5pt)
                        node[above left, text=orange] {$p_0$};
                \end{tikzpicture}
                \caption{$T$, its subset \colorbox{cyan!30}{$U_0$} and the extremal point $p_0$}
                \label{fig:T-X0-subset}
            \end{subfigure}
            
             \caption{Two views of $T$ and a version of $U_0$.}
            \label{fig:T-X0}
        \end{figure}
    \end{example}
    
    This naturally leads one to the question of whether there is a meaningful condition that allows to conclude equality of these two distances. In this case, the shared value is easy to compute, yet canonical.
    \\ This question was posed by Adams, Majhi, Manin, Virk and Zava in \cite{AdamsMajhiManinVirkZavaMetricGraphs}, along with a first answer to it.
    \\ In this paper, we improve upon the result given.

    To state our main theorem, we introduce the following notion:
    
    \begin{restatable}{definition}{extremalPointsDefinition}
        \label{definition:extremal-points}
        Let $X$ be a metric space and $A \subseteq X$ be a nonempty subset.
        \\ We define the \emph{set of extremal points of $A$ within $X$} as
        $$ \operatorname{Ext}(X, A) := \left\{x \in X : \dist_A(x) = \dH(A, X) \right\} . $$
    \end{restatable}

    We call a point in a metric graph a leaf if its degree within this graph is $1$ (also see \cref{subsubsec:metric-graphs}). We also denote by $\sys(G)$ the systole of a graph, defined as the minimal length of a non contractible closed curve in $G$. For a graph with no such curves (i. e. acyclic graphs) we define the systole to be infinite.

    \begin{theorem}
        \label{theorem:introduction-main}
        Let $G$ be a finite connected metric graph and let $U \subseteq G$ be a subset.
        \\ If $\operatorname{Ext}(G, U)$ is not fully contained within the set of leaves and $\dGH(G, U) < \frac{1}{12} \sys(G)$, then $\dGH(G, U) = \dH(G, U)$.
    \end{theorem}

    Looking at \cref{Example:ConstructionForHMuchBiggerThanGH}, we see that there is just one extremal point $p_0 \in \operatorname{Ext}(T, U_0)$ and that it is a leaf, making \cref{theorem:introduction-main} not applicable.
    
    \cref{theorem:introduction-main} is a generalization of Theorem 6.11 of \cite{AdamsMajhiManinVirkZavaMetricGraphs}. Compared to the latter, it increases the upper bound from $\frac{1}{24}e(G_0)$ where $e(G_0)$ is the length of the shortest edge contained in the graph's core (i. e. its minimal deformation retract) to $\frac{1}{12}\sys(G)$. Note that in every graph we have $\sys(G) \geq e(G_0)$ with equality only holding in certain cases. This makes the systole a more suitable quantity to relatively bound the Gromov-Hausdorff distance against.

    We prove \cref{theorem:introduction-main} as a special case of the more discerning \cref{theorem:main}, which unites Theorems 4.3 and 6.11 of \cite{AdamsMajhiManinVirkZavaMetricGraphs}. In it, we further distinguish between two types of non-leaf points. Extremal points of one such type then immediately imply $\dGH(G, U) = \dH(G, U)$.

    For the proof of \cref{theorem:main}, we employ a concept also introduced in \cite{AdamsMajhiManinVirkZavaMetricGraphs} which we have dubbed $\varepsilon$-continuizations. We then rely on basic results from Algebraic Topology as opposed to the more elementary construction used in \cite{AdamsMajhiManinVirkZavaMetricGraphs}, allowing for a more streamlined proof.

    \subsection*{Acknowledgements}
    The author would like to thank Adams, Majhi, Manin, Virk and Zava for posing the initial question and Alexander Lytchak for his valuable guidance and input.

\section{Prerequisites}
    \subsection{Assumed prerequisites}
        \subsubsection{Gromov--Hausdorff Distance}
            In this paper we assume the reader is familiar with the concept of the Gromov--Hausdorff distance and its various definitions:
            \begin{enumerate}[(i)]
                \item via the Hausdorff distance
                \begin{equation}
                    \dGH(X, Y) := \inf\left\{\dH^Z(X', Y') : X', Y' \text{ are isometric copies within a space } Z\right\} .
                \end{equation}
                \item via the distortion of correspondences
                \begin{equation}
                    \dGH(X, Y) = \frac{1}{2} \inf\left\{\dis(R) : R \text{ correspondence between } X \text{ and } Y\right\} .
                \end{equation}
                \item via the distortion and codistortion of two maps
                \begin{equation}
                    \dGH(X, Y) = \frac{1}{2} \inf\left\{ \max\left\{\dis(f), \dis(g), \codis(f, g)\right\} \text{ for } f : X \to Y, g : Y \to X \right\} ,
                \end{equation}
                wherein we use
                $$ \codis(f, g) = \sup\left\{\left|d_X(x, g(y)) - d_Y(f(x), y)\right| : x \in X, y \in Y \right\} . $$
            \end{enumerate}
            We recall that for two compact spaces all infima in this definition are attained for some shared isometric embedding, correspondence or pair of maps.
        \subsubsection{Metric Graphs}\label{subsubsec:metric-graphs}
            In this paper we only discuss connected metric graphs, a very narrow class of metric graphs which therefore does not require too formal a definition.
            \\ We define a \emph{finite connected metric graph} $G$ to be a compact length space in which every point $x$ has a neighbourhood that is isometric to a Euclidean Cone over a finite metric space equipped with its natural intrinsic length-metric via an isometry mapping $x$ to the tip of the cone.
            \\ We note that the number of points that induce this cone is unique with respect to $x$ and call it the \emph{degree} of $x$ in $G$.
            
            We say that a metric graph is equipped with a \emph{simplicial structure} when we fix a homeomorphism between it and a simplicial complex. In fixing a simplicial structure, we also fix which points in the graph we consider \emph{vertices}, namely the points corresponding to points of the $0$-skeleton of the simplicial complex. The choice of vertices of degree $2$ is highly dependent on the chosen structure while all points of degree $\neq 2$ are vertices regardless.
            \\ We also call two distinct vertices \emph{adjacent} with respect to a simplicial structure if they are part of the same edge.

            We call a point $g \in G$ of degree $1$ a \emph{leaf} and denote by \emph{$\partial G$} the set of leaves in $G$.
            \\ We call a point $g \in G$ a \emph{cut-point} if $G \setminus \{g\}$ is no longer connected.
            \\ We call a point $g \in G$ a \emph{cyclical point} if it is part of a simple cycle in $G$ (i.e. a non-contractible closed curve).
            \\ Note that a point that is neither a leaf nor a cut-point must be cyclical (though a point can be cyclical and a cut-point at the same time).

            In a finite metric graph $G$, there can only ever be finitely many simple cycles which all have positive lengths.
            \\ We define the \emph{systole} of a graph $G$ as the minimum length of a simple cycle and denote it \emph{$\sys(G)$}.
            \\ An immediate consequence is that two points that have distance $< \frac{1}{2} \sys(G)$ are connected by a unique geodesic.

            Note that a connected closed subset of a graph equipped with the subspace metric is generally not a metric graph again, as it need not be a length space. It is, however, homeomorphic to its own induced length space which is a finite connected metric graph. We therefore often apply concepts like the degree of points that only depend on topology and not the underlying metric to subsets of graphs as well.
    \subsection{Non-standard concepts and results}
        \subsubsection{Extremal Points For Subsets}
            \hypertarget{extremal-points-restatement}{}
            \begingroup
                \stepcounter{definition}
                \extremalPointsDefinition*
            \endgroup
            \begin{remark}
                If we assume a metric space $X$ to be compact and nonempty, we easily see that $\operatorname{Ext}(X, A) \neq \emptyset$ for every nonempty subset $A \subseteq X$.
                \\ We also see that $\operatorname{Ext}(X, A) = \operatorname{Ext}(X, \overline{A})$.
                \\ We also note that 
                $$ \dH(X, A) = \max_{x \in X} \dist(x, A) , $$
                and so the extremal points are exactly the points in $X$ attaining the maximal distance to $A$.
            \end{remark}
        \subsubsection{Separation distance}
            
			\begin{definition} \label{definition:separation distance}
				For a metric space $X$, we define its \emph{separation distance} as
                \begin{equation}
                    \Sep(X) := \sup \left\{r > 0 : \exists_{\emptyset \neq A, A' \subseteq X} : A \cup A' = X, \dist(A, A') > r\right\} ,
                \end{equation}
                wherein we use $\sup\emptyset = 0$.
			\end{definition}
            
			\begin{remark}
				In the definition of separation distance we require neither disjointness nor openness of the sets $A$ and $A'$.
				\\ For a metric space with $\Sep(X) > 0$, we may as well require openness and disjointness, as $\dist(A, A') > r > 0$ implies $A \cap A' = \emptyset$.
                Together with $A \cup A' = X$ we see that for $0 < \rho < r < \frac{\Sep(X)}{2}$ we have $A = B_\rho(A)$ and $A' = B_\rho(A')$, implying openness.
			\end{remark}

            \begin{remark}
				Consider a connected space $X$ and any two nonempty sets $A, A' \subseteq X$ such that $A \cup A' = X$.
				\\ Note that the closures $\overline{A}$ and $\overline{A'}$ of $A$ and $A'$ then also satisfy these properties.
                \\ In addition, they are closed sets satisfying $\overline{A} \cup \overline{A'} = X$, so by connectedness of $X$ they cannot be disjoint but intersect, implying
				$$ \dist(A, A') = \dist(\overline{A}, \overline{A'}) = 0 . $$
				Since $A$ and $A'$ were arbitrary with respect to the conditions of \cref{definition:separation distance} we conclude $\Sep(X) = 0$.
			\end{remark}
			\begin{lemma}
				The separation distance is a $2$-Lipschitz property of bounded metric spaces with respect to the Gromov--Hausdorff distance, i.e. for any two bounded metric spaces $X$ and $\tilde{X}$ we have
				$$ \left|\Sep(X) - \Sep(\tilde{X})\right| \leq 2 \cdot \dGH(X, \tilde{X}) . $$
			\end{lemma}
			\begin{proof} \; \\ \noindent
				To prove this, note that it suffices to show the following:				
				\begin{adjustwidth}{2em}{0em}
					Assume that for some $r > 0$ in a bounded space $X$ there exist nonempty sets $A$ and $A'$ such that $A \cup A' = X$ and $\dist(A, A') > r$.
					\\ Then in any bounded space $\tilde{X}$ there exist nonempty sets $\tilde{A}$ and $\tilde{A}'$ such that $\tilde{A} \cup \tilde{A}' = \tilde{X}$ and $\dist(\tilde{A}, \tilde{A}') \geq r - 2 \dGH(X, \tilde{X})$.
				\end{adjustwidth}
				Fix $\varepsilon > 0$ such that $\dist(A, A') - \varepsilon > r$.
				\\ Now choose a correspondence $R$ between $X$ and $\tilde{X}$ such that $\dis(R) < 2 \dGH(X, \tilde{X}) + \varepsilon$.
				\\ Setting $\tilde{A} = R[A]$ and $\tilde{A}' = R[A']$ we see that $\tilde{A} \cup \tilde{A}' = \tilde{X}$.
                \\ Also, for $\tilde{a} \in \tilde{A}$ and $\tilde{a}' \in \tilde{A}'$ find points $a \in A$ and $a' \in A'$ such that $(a, \tilde{a}), (a', \tilde{a}') \in R$.
                \\ We then have 
				$$ d(\tilde{a}, \tilde{a}') \geq d(a, a') - \dis(R) > \dist(A, A') - \varepsilon - 2\dGH(X, \tilde{X}) > r - 2\dGH(X, \tilde{X}) $$ 
				which in particular implies $\dist(\tilde{A}, \tilde{A}') \geq r - 2\dGH(X, \tilde{X})$.
			\end{proof}
		\subsubsection{\texorpdfstring{$\varepsilon$-continuizations}{epsilon-continuizations}}
            \begin{construction}[$\varepsilon$-continuizations] \label{Construction:Epsiloncontinuizations}
                Let $G$ and $G'$ be finite connected metric graphs and let $A \subseteq G$ be a compact connected subset of $G$ equipped with the subspace-metric.
                \\ Now, consider a map $f : A \to G'$ and $\varepsilon > 0$.
                \\ We say that $\thickbar{f} : A \to G'$ is an \emph{$\varepsilon$-continuization} of $f$ (with \emph{vertex-points} $V \subseteq A$) if it results from $f$ in the following way:
                \begin{itemize}
                    \item Let $V_0$ be the finite set of all points in $G$ that have degree $\neq 2$ as points of $G$ or as points in $A$.
                    We fix a finite $\varepsilon$-net $V \supseteq V_0$ and a simplicial structure for $G$ such that all points in $V$ are vertices.
                    \item We set $\thickbar{f}(v) = f(v)$ for any $v \in V$.
                    \item We map geodesic segments $[v, v'] \subseteq A$ for adjacent $v, v' \in V$ affinely to some geodesic segment from $f(v)$ to $f(v')$.
                \end{itemize}
                Note that generally, this segment does not have to be unique. However, once we make a choice of segment, $\thickbar{f}$ is well-defined as every point lies in $V$ or on a unique geodesic segment between adjacent vertices in $V$.
            \end{construction}
            \begin{lemma} \label{Lemma:PropertiesOfEpsiloncontinuizations}
                Let $f : G \to G'$ be a map between two finite metric graphs, $\varepsilon > 0$, and let $\thickbar{f}$ be an $\varepsilon$-continuization of $f$ with a vertex set $V$. Then,
                \begin{enumerate}[(i)]
                    \item $$ \thickbar{f}(G) \subseteq f(G)^{\left(\frac{\dis(f) + \varepsilon}{2}\right)} $$
                    \item $$ \dis\left(\thickbar{f}\right) \leq 2 \dis(f) + 2 \varepsilon . $$
                \end{enumerate}
            \end{lemma}
            \begin{proof} \; \\ \noindent
                \textbf{Part 1)} (i)
                \\ For a point $x \in G$ we can find vertex-points $v, v' \in V$ such that $d(v, v') \leq \varepsilon$, $x \in [v, v']$ and $d(v, x) \leq \frac{1}{2} d(v, v')$. Since $\thickbar{f}$ results from affine parametrization along the edges, we then get
                \begin{equation}
                    d(\thickbar{f}(v), \thickbar{f}(x)) \leq \frac{1}{2} d(\thickbar{f}(v), \thickbar{f}(v')) = \frac{1}{2} d(f(v), f(v')) \leq \frac{1}{2}\left(d(v, v') + \dis(f)\right) < \frac{1}{2}\left(\dis(f) + \varepsilon\right) .
                \end{equation}
                As $\thickbar{f}(v) = f(v) \in f(G)$, this proves (i).
                \\ \textbf{Part 2)} (ii)
                \\ For points $x, \tilde{x} \in G$ we can find vertex-points $v, v', \tilde{v}, \tilde{v}' \in V$ such that $d(v, v') < \varepsilon$, $d(\tilde{v}, \tilde{v}') < \varepsilon$, $x \in [v, v']$ and $\tilde{x} \in [\tilde{v}, \tilde{v}']$ with $d(v, x) \leq \frac{1}{2} d(v, v')$ and $d(\tilde{v}, \tilde{x}) \leq \frac{1}{2} d(\tilde{v}, \tilde{v}')$.
                Applying what we showed in the proof of (i) to this, we get
                \begin{flalign*}
                    & \left|d(\thickbar{f}(x), \thickbar{f}(\tilde{x})) - d(x, \tilde{x})\right|
                    \\ \leq & \underbrace{\left|d(\thickbar{f}(v), \thickbar{f}(\tilde{v})) - d(v, \tilde{v})\right|}_{\leq \dis(f)} + \underbrace{d(\thickbar{f}(v), \thickbar{f}(x))}_{< \frac{\dis(f) + \varepsilon}{2}} + \underbrace{d(\thickbar{f}(\tilde{v}), \thickbar{f}(\tilde{x}))}_{< \frac{\dis(f) + \varepsilon}{2}} + \underbrace{d(v, x)}_{< \frac{\varepsilon}{2}} + \underbrace{d(\tilde{v}, \tilde{x})}_{< \frac{\varepsilon}{2}}
                    \\ \leq & 2 \dis(f) + 2 \varepsilon
                \end{flalign*}
                Since $x, \tilde{x} \in G$ were arbitrary, we have $\dis\left(\thickbar{f}\right) \leq 2 \dis(f) + 2 \varepsilon$, proving (ii).
            \end{proof}

    \subsection{Implicitly used results}
        \subsubsection{Uniqueness of geodesics implies continuous dependence on endpoints}
            Note that unless otherwise stated in the following, we always assume geodesics to be parameterized on $[0, 1]$ proportionally to arclength.
            \begin{lemma}[Continuously varying geodesics {\cite[Lemma~2.5]{LangLengthSpaces}}]
                \label{Lemma:UniqueGeodesicsDependContinuouslyOnEndpoints}
                Let $X$ be a proper metric space, and let $\sigma : [0,1] \to X$ be a geodesic such that $\sigma([0,1])$ is the only geodesic segment connecting $\sigma(0)$ and $\sigma(1)$.

                Suppose further that $\sigma_k : [0,1] \to X$ is a sequence of geodesics such that $\sigma_k(0) \to \sigma(0)$ and $\sigma_k(1) \to \sigma(1)$.
                Then $\sigma_k$ converges uniformly to $\sigma$, i.e.
                \[
                    \sup_{t \in [0,1]} d_X\bigl(\sigma_k(t), \sigma(t)\bigr) \to 0 .
                \]
            \end{lemma}
            The proof can be found in \cite[Lemma~2.5]{LangLengthSpaces}.
            
            As an immediate consequence we have:

            \begin{corollary} \label{Corollary:HomotopyFromUniqueGeodesics}
                Let $X$ and $Y$ be two proper length spaces and let $f, g : X \to Y$ be two continuous maps such that for all $x \in X$ there exists a unique geodesic $\sigma_x$ between $f(x)$ and $g(x)$.
                \\ Then there exists a homotopy between $f$ and $g$ given by:
                \begin{equation}
                    h : [0, 1] \times X \to Y, (t, x) \mapsto \sigma_x(t) .
                \end{equation} 
            \end{corollary}

        \subsubsection{Cellular homology of a graph}
            We shall use the cellular homology theory with coefficients in $\bbZ/2\bbZ$, in particular the homotopy-invariance and the Euler characteristic for metric graphs:
            $$ \beta_0(G) - \beta_1(G) = |V| - |E| $$
            where $\beta_0$ and $\beta_1$ are Betti numbers and $V$ and $E$ are the vertex and edge sets for a fixed realization of $G$ as a simplicial complex.
            One quickly sees that they justify the following result:
            \begin{lemma} \label{Lemma:BettiNumbersForMapsOntoProperSubsetsOfGraphs}
                Let $G$ be a connected finite metric graph, $g \in G$ a point that is neither a leaf nor a cut-point in $G$ and $U \subseteq G \setminus \{g\} \subseteq G$ a compact connected subset.
                \\ Then $\beta_1(U) < \beta_1(G)$.
            \end{lemma}
            \begin{proof}
                We fix a simplicial structure for $G$ in which $g$ is a vertex and w.l.o.g. choose $\varepsilon$ smaller than all the edge-lengths of edges incident to $g$ and small enough that $U \subseteq G \setminus B_{\varepsilon}(g)$ using the compactness of $U$.
                \\ W.l.o.g. we also make all elements of $\partial B(g, \varepsilon)$ vertices in $G$ by potentially subdividing some edges.
                \\ We then see that removing $B(g, \varepsilon)$ from $G$ to form $G' := G \setminus B(g, \varepsilon)$ removes exactly the vertex $g$ and all its incident edges (at least $2$ edges because $g$ is a cyclical point).
                \\ It therefore decreases $|V|$ by $1$, decreases $|E|$ by $k \geq 2$ and keeps $\beta_0$ constant.
                Applying the Euler characteristic formula for metric graphs, we therefore see
                $$
                \beta_1(G')
                = \underbrace{\beta_0(G')}_{= \beta_0(G)}
                + \underbrace{|E(G')|}_{\leq |E(G)| - 2}
                \underbrace{- |V(G')|}_{\quad = -|V(G)| + 1}
                < \beta_0(G) + |E(G)| - |V(G)|
                = \beta_1(G) .
                $$
                In particular, we can now use the fact that $U$ is a compact connected subset of the $1$-dimensional simplicial complex $G'$, and so with the right simplicial structure for $G'$ we easily see that the injection $\operatorname{id} : U \to G'$ induces an injection $H_1(U) \to H_1(G')$, in particular implying $\beta_1(U) \leq \beta_1(G')$. Details behind this argument can be found in \cite[Section~2.1]{HatcherAlgebraicTopology}.
                \\ Using this, we finally conclude:
                $$ \beta_1(U) \leq \beta_1(G') < \beta_1(G) . $$
            \end{proof}
\section{Statement and Proof}
        The main statement of this paper is the following:
        \begin{restatable}{theorem}{mainTheorem}
            \label{theorem:main}
            Let $G$ be a finite metric graph and let $U \subseteq G$ be a subset.
            \begin{enumerate}[(i)]
                \item If $\operatorname{Ext}(G, U)$ contains any cut-point of $G$, then $\dGH(G, U) = \dH(G, U)$.
                \item If $\operatorname{Ext}(G, U)$ contains a cyclical point of $G$ and we have
                \[
                    \dGH(G, U) < \frac{1}{12} \sys(G) ,
                \]
                then $\dGH(G, U) = \dH(G, U)$.
            \end{enumerate}
        \end{restatable}
        \begin{proof}[Proof]
            In the case $\dH(G, U) = 0$ the statements obviously hold and thus, we may assume that $\dH(G, U) > 0$. By considering its closure $\overline{U}$ instead of $U$, we may also assume that $U$ is closed and therefore compact.
            \\ We now fix an extremal point $g \in \operatorname{Ext}(G, U)$.

            \textbf{Case 1)} $g$ is a cut-point of $G$.
            \\ Consider at least two of the connected components of $G \setminus \{g\}$ and assume that they do not both contain elements of $U$.
            \\ Clearly, moving $g$ into a component that does not contain elements of $U$ would then increase the distance to all other components of $U$, contradicting the fact that $g$ is an extremal point maximizing the distance to $U$.
            \\ Thus, in any connected component of $G \setminus \{g\}$ there must be an element of $U$, allowing us to partition $U$ into the intersection $U'$ of $U$ with one such component and $U'' := U \setminus U'$.
            \\ The shortest path between elements of $U'$ and $U''$ must always go through $g$ and therefore be of length $\geq \dist(g, U') + \dist(g, U'') \geq 2\dH(G, U)$.
            \\ In particular, this proves that $\operatorname{Sep}(U) \geq 2\dH(G, U)$, whereas $\operatorname{Sep}(G) = 0$ as $G$ is connected.
            \\ From this we conclude
            $$ 2 \dGH(G, U) \geq \left|\Sep(G) - \Sep(U)\right| = \Sep(U) \geq 2\dH(G, U) . $$

            \textbf{Case 2)} $g$ is neither a leaf nor a cut-point and $\dGH(G, U) < \frac{1}{12} \sys(G)$.
            
            \textbf{Step 1)}
            \\ We fix a map $f : G \to U$ of distortion $\dis(f) \leq 2 \dGH(G, U)$.
            \\ For contradiction, we assume that $\dGH(G, U) < \dH(G, U)$.
            \\ We are then able to choose a value $\varepsilon > 0$ such that 
            $$ \underbrace{\frac{\dis(f) + \varepsilon}{2}}_{=: r} < \underbrace{\frac{\dis(f) + 2\varepsilon}{2}}_{= r + \frac{\varepsilon}{2}} < \min\left\{\dH(G, U), \frac{1}{12}\sys(G)\right\} . $$
            We fix an $\varepsilon$-continuization $\thickbar{f} : G \to G$ of $f$ on a vertex-set $V$ and equip $G$ with the corresponding simplicial structure.
            We then have by \cref{Lemma:PropertiesOfEpsiloncontinuizations} that 
            $$ \dis(\thickbar{f}) \leq 4r < 4r + 2 \varepsilon < 4 \min\left\{\dH(G, U), \frac{1}{12}\sys(G)\right\} \leq \frac{1}{3} \sys(G) $$
            and
            $$ \thickbar{f}(G) \subseteq B_{r}(U) \subseteq G \setminus \{g\} . $$
            In particular, since $g$ is not a cut-point of $G$ we can apply \cref{Lemma:BettiNumbersForMapsOntoProperSubsetsOfGraphs} and see that for $W := \thickbar{f}(G) \subseteq G \setminus \{g\} \subseteq G$ we have $\beta_1(W) < \beta_1(G)$.
            \\ From this we can in particular conclude $H_1(\thickbar{f}) : H_1(G) \to H_1(W)$ is not injective.

            \textbf{Step 2)}
            \\ On the compact set $W := \thickbar{f}(G) \subseteq G$ choose a finite $\varepsilon$-net $V \subseteq W$ containing all vertices of degree $\neq 2$ in $W$ or $G$ and fix the corresponding simplicial structure.
            We now choose a right-inverse $a : W \to G$ of the surjection $\thickbar{f} : G \to W$ and from it construct an $\varepsilon$-continuization $\thickbar{a} : W \to G$ for which we will show that
            \begin{equation} \label{eq:RequirementForFunctiona}
                \dH(\thickbar{a}(w), \thickbar{f}^{-1}(w)) \leq \frac{1}{2}\sys(G) - \varepsilon \qquad \text{ for } w \in W .
            \end{equation}

            For any point $w \in W$ we can now consider the edge-segment $[v, v']$ containing $w$ (or one of them if $w$ is a vertex), such that $d(w, v) \leq \frac{1}{2} d(v, v')$ and as we have seen in the proof of \cref{Lemma:PropertiesOfEpsiloncontinuizations} we have
            \begin{equation}
                d(\thickbar{a}(w), \thickbar{a}(v)) \leq \frac{\dis(a) + \varepsilon}{2} \leq \frac{\dis(\thickbar{f}) + \varepsilon}{2} \leq 2r + \frac{\varepsilon}{2} .
            \end{equation}
            Also, using $\thickbar{a}(v) = a(v) \in \thickbar{f}^{-1}(v)$ we have
            \begin{equation}
                \dH(\thickbar{a}(v), \thickbar{f}^{-1}(w)) 
                \leq \diam(\thickbar{f}^{-1}(\{v, w\}))
                \leq \dis(\thickbar{f}) + \diam(\{v, w\})
                \leq 4r + \frac{\varepsilon}{2} .
            \end{equation}
            Combining these inequalities yields
            \begin{equation}
                \dH(\thickbar{a}(w), \thickbar{f}^{-1}(w)) 
                \leq d(\thickbar{a}(w), \thickbar{a}(v)) + \dH(\thickbar{a}(v), \thickbar{f}^{-1}(w))
                \leq 6r + \varepsilon
                < \frac{1}{2}\sys(G) - \varepsilon .
            \end{equation}
            As we have now shown that $\thickbar{a}$ satisfies \cref{eq:RequirementForFunctiona}, for every point $w \in W$ there exists a unique geodesic from $\thickbar{a}(w)$ to any point in $\thickbar{f}^{-1}(w)$ of length $\leq \frac{1}{2}\sys(G) - \varepsilon$.
            \\ In particular, we see that for every $x \in G$ there exists a unique such geodesic between $x = \operatorname{id}_{G}(x) \in \thickbar{f}^{-1}(\thickbar{f}(x))$ and $(\thickbar{a} \circ \thickbar{f})(x)$.
            \\ The existence and uniqueness of these geodesics by \cref{Corollary:HomotopyFromUniqueGeodesics} imply a homotopy between $\operatorname{id}_{G}$ and $(\thickbar{a} \circ \thickbar{f})$.
            \\ This then implies
            $$ \operatorname{id}_{H_1(G)} = H_1(\thickbar{a} \circ \thickbar{f}) = H_1(\thickbar{a}) \circ H_1(\thickbar{f}) $$
            while in Step 1) we showed that $H_1(\thickbar{f})$ is not injective.
            \\ This finally gives the desired contradiction.
        \end{proof}

\section{Open Questions}
    % \subsection{The origin of the factor $\frac{1}{12}$}
        The biggest open question is, of course, whether the constant factor $\frac{1}{12}$ in \cref{theorem:main} can be improved. 
        We note that using the approach taken in this paper, it is the composition of three factors $\frac{1}{2}$, $\frac{1}{2}$ and $\frac{1}{3}$. Counterexamples proving that these bounds cannot generally be improved in either one of the three situations exist.
        However, we are not certain that these three counterexamples can coincide in an example where $\dGH(G, U) < \dH(G, U)$ while $\dGH(G, U) = \frac{1}{12}\sys(G)$.
        % \begin{itemize}
        %     \item The first factor of $2$ comes from the definition of the Gromov-Hausdorff distance via correspondences. Here, we use the Gromov-Hausdorff distance to consider a function $f : G \to U$ of distortion $\leq 2 \dGH(G, U)$.
        %     \item A further factor of $2$ comes from the distortion formula for $\varepsilon$-continuizations, giving us a continuous function $\thickbar{f} : G \to U$ of distortion $\leq 4 \dGH(G, U)$.
        %     \item The last factor of $3$ comes from the fact that below a threshold of the form $\dis(\thickbar{f}) < \frac{1}{3} \sys(G)$, we can prove that $\thickbar{f}$ is a homotopy equivalence.  
        % \end{itemize}
        
    % \subsection{Improvements upon the result}
        % The different ideas used are each optimal on their own with counterexamples showing that at no point we could simply swap in a better factor. More precisely,
        % \begin{itemize}
        %     \item There are subsets $U$ of metric graphs $G$ such that $\dis(f) \geq 2 \dGH(G, U)$ for every map $f : G \to U$.
        %     \item There are maps $f : G \to U$ such that their $\varepsilon$-continuizations have distortion up to $2(\dis(f) + \varepsilon)$.
        %     \item There is a continuous map $f : S^1 \to T$ where $T$ is a tripod with three edges of length $\frac{\pi}{3}$ such that $\dis(f) = \frac{2 \pi}{3} = \frac{1}{3}\sys(G)$. Such an $f$ then obviously does not preserve homology.
        % \end{itemize}
        % However, a counterexample that proves all three factors optimal at the same time has yet to be found.
        
        To discuss the possibilities for improvement, we consider a plot of all the possible pairs $(\dH(G, U), \dGH(G, U))$ for graphs $G$ and subsets $U$ where $\operatorname{Ext}(G, U) \nsubseteq \partial G$ (rescaled according to their systole):
        \begin{figure}[H]
            \centering
            \begin{tikzpicture}
                % \fill[orange!45] (0,0) -- (4,4) -- (4,0) -- cycle;
                % \fill[orange!22.5] (4,0) -- (4,4) -- (6.5,4) -- (6.5,0) -- cycle;
                \fill[pattern=north west lines, pattern color=gray!60]
                    (0,0) -- (0,6.5) -- (6.5,6.5) -- cycle;
                \fill[pattern=north west lines, pattern color=red]
                    (0,0) -- (2,2) -- (6.5,2) -- (6.5,0) -- cycle;
                \draw[->] (0,0) -- (6.5,0)
                    node[below] {$\dH(G, U)$};
                \draw[->] (0,0) -- (0,6.5)
                    node[left] {$\dGH(G, U)$};
                \draw (2,0.08) -- (2,-0.08)
                    node[below] {$\frac{1}{12}\sys(G)$};
                \draw (4,0.08) -- (4,-0.08)
                    node[below] {$\frac{1}{6}\sys(G)$};
                \draw (0.08,2) -- (-0.08,2)
                    node[left] {$\frac{1}{12}\sys(G)$};
                \draw (0.08,4) -- (-0.08,4)
                    node[left] {$\frac{1}{6}\sys(G)$};
                \draw[black, thick] (0,0) -- (6.5,6.5)
                    node[pos=0.82, above left, font=\small]
                    {$\dGH(G, U) = \dH(G, U)$};
                \draw[black, line width=2pt, line cap=round] (0,0) -- (2,2);
                \draw[blue, line width=2pt, line cap=round] (4,4) -- (6,4);
                \draw[red, thick, dashed] (2,2) -- (6.5,2);
            \end{tikzpicture}
            \caption{Visualizing Bounds Guaranteeing Equality}
            \label{fig:VisualizingBoundsForEquality}
        \end{figure}
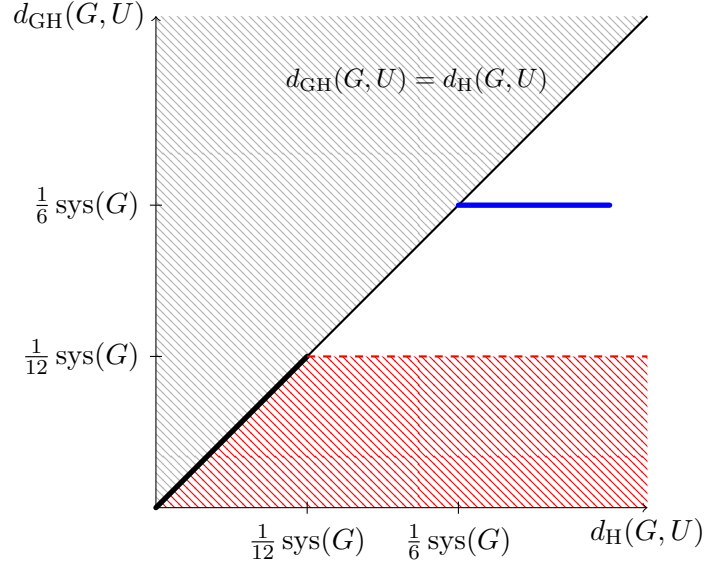
        We have marked everything above the diagonal impossible (\patternswatch{north west lines}{gray!60}), as $\dGH(G, U) \leq \dH(G, U)$, by hatching it in gray.
        \\ \cref{theorem:main} now guarantees that points in the red hatched area (\patternswatch{north west lines}{red}) are also impossible.
        \\ A converse bound was given in Theorem 5.2 of \cite{AdamsMajhiManinVirkZavaMetricGraphs} for the unit circle $G = S^1$ by explicitly constructing a correspondence to certain subsets $U_\varepsilon \subseteq S^1$ that shows 
        $$ \dGH(S^1, U_\varepsilon) = \frac{\pi}{3} = \frac{1}{6} \sys(S^1) < \frac{\pi}{3} + \varepsilon = \dH(S^1, U_\varepsilon) \quad \text{ for } \quad \varepsilon \in \left(0, \frac{\pi}{6}\right] $$
        and we have marked the corresponding points in \cref{fig:VisualizingBoundsForEquality} blue (\colorswatch{blue}).
        \\ There are now two ways of improving the results of this paper: Concluding equality from $\dH(G, U) < c \cdot \sys(G)$ for $c \in \left(\frac{1}{12}, \frac{1}{6}\right]$ would prevent inequality to the left of a certain vertical line while concluding equality from $\dGH(G, U) < c \cdot \sys(G)$ would prevent inequality below a certain horizontal line. 
        \\ Given the blue line of examples of inequality and the red hatched area, where we can already guarantee equality, there must be values $c_{\operatorname{H}}, c_{\operatorname{GH}} \in \left[\frac{1}{12}, \frac{1}{6}\right]$ that are infimal with respect to the conditions that there are no inequalities for $\dH(G, U) < c_{\operatorname{H}} \sys(G)$ or $\dGH(G, U) < c_{\operatorname{GH}} \sys(G)$ respectively, bringing us to the question:
        \begin{question}
            What are the values $c_{\operatorname{H}}$ and $c_{\operatorname{GH}}$? 
            % Are they equal? What happens around $\dH(G, U) = c_{\operatorname{H}}\sys(G)$ or $\dGH(G, U) = c_{\operatorname{GH}} \sys(G)$?
        \end{question}
        % A related question would then be to ask:
        % \begin{question}
        %     What happens around $\dH(G, U) = c_{\operatorname{H}}\sys(G)$ or $\dGH(G, U) = c_{\operatorname{GH}} \sys(G)$?
        % \end{question}
        % Lastly, we chose the systole as a way to quantify the size of a graph as it was fitting for the approach taken in the proof. However, it is one of several approaches, leading us to ask:
        % \begin{question}
        %     Are there quantities that yield better or independent conditions for equality, as opposed to the systole?
        % \end{question}
    \clearpage
    \bibliographystyle{unsrt}
    \bibliography{sample}
\end{document}